\documentclass[12pt,letterpaper]{article}
\usepackage{amsmath}
\usepackage{amssymb}
\usepackage{amsfonts}
\usepackage{amsthm}
\usepackage{mathtools}

\usepackage[T2A,OT1]{fontenc}
\usepackage[utf8]{inputenc}
\usepackage[russian,english]{babel}

 \newtheorem{thm}{Theorem}[section]
 \newtheorem{cor}[thm]{Corollary}
 \newtheorem{lem}[thm]{Lemma}
 \newtheorem{prop}[thm]{Proposition}
 
 \theoremstyle{definition}
 
 \newtheorem*{defn}{Definition}
 
\theoremstyle{notation}

 \theoremstyle{remark}
 \newtheorem*{rem}{Remark}
 
 \theoremstyle{langrange formula}
 \newtheorem*{lan}{Langrange Formula}

 \theoremstyle{jacobi formula}
 \newtheorem*{jac}{Jacobi Formula} 
 
 \theoremstyle{fact 1}
 \newtheorem*{fact1}{Fact 1} 

 \theoremstyle{fact 2}
 \newtheorem*{fact2}{Fact 2}

 \numberwithin{equation}{section}
 
 \allowdisplaybreaks
 
\newcommand{\V}{\mathcal{V}}

\title{Sums of Squared Distances between Points on a Unit $n$-Sphere}
\author{Jessica N. Copher}

\begin{document}
\maketitle

\begin{abstract}
In this paper, we prove two theorems concerning the sums of squared distances between points on a unit $n$-sphere that generalize two facts previously known about the case where the points are the vertices of a regular polygon.
  The first theorem is that, given a multiset of $V$ points on a unit $n$-sphere, the sum of the squared distances between these points is $V^2(1-d^2)$ where $d$ is the distance between the centroid of the points and the center of the unit $n$-sphere (for any $n \geq 2$).  
  The second is that, given a finite set of points on the unit $n$-sphere centered at the origin such that the point set is symmetric about the origin and the symmetry group of the point set acts transitively on the set, the sum of the squared distinct distances between these points is $2k+2$ where $k$ is the number of distinct distances between the points (for any $n \geq 2$).
    Using the first theorem, we find a new way to calculate the potential energy function of a finite normalized frame.
\end{abstract}

\section{Introduction}

The motivation for this research came from the following two facts. 
Consider the set of points $\mathcal{V}$ consisting of the vertices of a regular polygon with $E$ edges that is inscribed in a unit circle.
Then:
\begin{fact1}
The sum of the squared distances between pairs of vertices in $\mathcal{V}$ equals $E^{2}$ (see \cite{Mustonen} and \cite[Solution 6.73]{Prasolov}).
\end{fact1}
\begin{fact2}
The sum of the squared \emph{distinct} distances between pairs of vertices in $\mathcal{V}$
equals:
\begin{enumerate}
\item $E$ (when $E$ is odd) \cite{MorleyHarding}
\item some integer (when $E$ is even) \cite[pp. 490-491]{Kappraff2002}.
\end{enumerate}
\end{fact2}

The two main results of this paper generalize these two facts. The generalization of Fact 1 is a formula for the sum of the squared distances between pairs of points belonging to a multiset $\V$ of points on a unit $n$-sphere.  
 The generalization of Fact 2 is a formula for the sum of the squared distinct distances between pairs of points in any certain ``very symmetric'' set $\mathcal{V}$ of points on a unit $n$-sphere.
 
\begin{rem}
In the literature, one finds similar facts regarding the \emph{products} of squared distances and squared distinct distances between pairs of vertices of regular polygons (see \cite{Mustonen} and \cite{Kappraff2002}, respectively).  For a discussion of how these facts generalize, see \cite{Copher}.
\end{rem}

\section{Sum of Squared Distances}

Before we generalize Fact 1, we need a bit of terminology.  
By a \emph{unit $n$-sphere}, we mean the set of all points in $\mathbb{R}^n$ at a distance of 1 from a fixed point.

\begin{defn}
Let $\mathcal{V}$ be a multiset of points on an $n$-sphere. Define a \emph{chord of $\mathcal{V}$} to be a line segment whose endpoints belong to $\mathcal{V}$.
\end{defn}

\begin{defn}
Let $\V = \{ \mathbf{P_1}, \mathbf{P_2}, \ldots, \mathbf{P_V} \}$ where $\mathbf{P_i}$ has Cartesian coordinates $(x_{1,i}, x_{2,i}, \ldots, x_{n,i})$.  Then the centroid of $\V$ (assuming $\mathbf{P_i}$ has ``unit mass'') is the point given by $\big ( \frac{1}{V} \sum_{i=1}^V x_{1,i}, \frac{1}{V} \sum_{i=1}^V x_{2,i}, \ldots, \frac{1}{V} \sum_{i=1}^V x_{n,i} \big )$.
\end{defn}

We can now generalize Fact 1 as follows.

\begin{thm}
\label{Thm:sumoallchords}
 Let $\V$ be a multiset of $V$ points on a unit $n$-sphere, and let $\mathcal{C}$ be the multiset of the lengths of  all the chords between them.
Then: \[ \sum_{c \, \in \, \mathcal{C}}c^2= V^{2}(1-d^2) \]
where $d$ is the distance between the centroid of $\V$ and the center of the unit $n$-sphere.
\end{thm}

This formula has an interesting interpretation in terms of expected value:  Choose two (not necessarily distinct) points $\mathbf{P}$ and $\mathbf{Q}$ uniformly randomly from $\V$.  Then the expected value of the squared length of the chord $\mathbf{PQ}$ is $1-d^2$.  The formula also has a noteworthy corollary:

\begin{cor}
Let $\V$ be a set of $V$ points on a unit $n$-sphere, and let $\mathcal{C}$ be the corresponding multiset of chord lengths.
Then $\sum_{c \, \in \, \mathcal{C}}c^2 \leq V^{2}$ with equality holding if and only if the centroid of $\V$ coincides with the center of the $n$-sphere.
\end{cor}

Now we give two proofs of Theorem \ref{Thm:sumoallchords}.  In the first proof, it is convenient to use hyperspherical coordinates.  For a point $\mathbf{P}$, the radial coordinate $\rho$ is the distance from the origin $\mathbf{O}$ to $\mathbf{P}$.  The angular coordinates $\phi^j$ for $j=1,2,\ldots,n-1$ measure the angle between the positive $x_j$-axis and the projection of the ray $\mathbf{OP}$ onto the $x_j \cdots x_n$-space.  Thus, the Cartesian coordinates of $\mathbf{P}$ are
\begin{multline*}
 \big (\rho \cos \phi^1, \rho \cos \phi^2 \sin \phi^1, \rho \cos \phi^3 \sin \phi^2 \sin \phi^1, \\
 \ldots, \rho \cos \phi^{n-1} \prod_{l=1}^{n-2} \sin \phi^l, \rho \prod_{l=1}^{n-1} \sin \phi^l \big ).
\end{multline*}

It will also be helpful to use the following trigonometric identity.

\begin{lem}
\label{Lem:trigidentity}
For any angles $\phi^1,\ldots, \phi^{n-1}$,
\[ \sum_{k=1}^{n-1} \cos^2 \phi^k \prod_{l=1}^{k-1} \sin^2 \phi^l + \prod_{l=1}^{n-1} \sin^2 \phi^l = 1\]
\end{lem}
\begin{proof}[Proof of Lemma \ref{Lem:trigidentity}]
We rearrange the left-hand side and use the identity $\sin^2 \phi + \cos^2 \phi = 1$ inductively: 
\begin{multline*}
\cos^2 \phi^1 + \sin^2 \phi^1 ( \cos^2 \phi^2  +  \sin^2 \phi^2 (  \cos^2 \phi^3 + \cdots \\
+ \sin^2 \phi^{n-3}( \cos^2 \phi^{n-2} + \sin^2 \phi^{n-2}(\cos^2 \phi^{n-1} + \sin^2 \phi^{n-1})) \cdots )) = 1
\end{multline*}
\end{proof}

\begin{proof}[First Proof of Theorem \ref{Thm:sumoallchords}]
Without loss of generality, suppose the $n$-sphere is centered at the origin.  Let $\V = \{ \mathbf{P_1}, \mathbf{P_2}, \ldots, \mathbf{P_V} \}$ where $\mathbf{P_i}$ has hyperspherical coordinates $(1, \phi_i^1, \ldots, \phi_i^{n-1})$.
To prove the identity, we will show that both sides are equal to 
\[
V^2 - V - \sum_{i \neq j \in [V]} \Bigg ( \bigg ( \sum_{k=1}^{n-1} \cos \phi_i^k \cos \phi_j^k \prod_{l = 1}^{k-1} \sin \phi_i^l \sin \phi_j^l \bigg ) + \bigg ( \prod_{l=1}^{n-1} \sin \phi_i^l \sin \phi_j^l \bigg ) \Bigg ). 
\]
Starting with the right-hand side, we have:
\begin{align*}
\MoveEqLeft V^2(1-d^2)\\
&
\!\begin{multlined}[t][12.05cm]
=V^2 \bigg (1- \bigg | \bigg | \bigg \langle \frac{1}{V} \sum_{i=1}^V \cos \phi_i^1, \frac{1}{V} \sum_{i=1}^V \cos \phi_i^2 \sin \phi_i^1,\\
\ldots, \frac{1}{V} \sum_{i=1}^V \cos \phi_i^{n-1} \prod_{l=1}^{n-2} \sin \phi_i^l, \frac{1}{V} \sum_{i=1}^V \prod_{l=1}^{n-1} \sin \phi_i^l \bigg \rangle \bigg | \bigg |^2 \bigg ) 
\end{multlined}
\\
&=V^2 \bigg (1 - \frac{1}{V^2} \sum_{k=1}^{n-1} \bigg (\sum_{i=1}^V \cos \phi_i^k \prod_{l=1}^{k-1} \sin \phi_i^l \bigg )^2 - \frac{1}{V^2} \Big (\sum_{i=1}^V \prod_{l=1}^{n-1} \sin \phi_i^l \Big )^2 \bigg ) \\
&
\!\begin{multlined}[t][12.05cm]
=V^2 - \sum_{k=1}^{n-1} \bigg (\sum_{i=1}^V \cos^2 \phi_i^k \prod_{l=1}^{k-1} \sin^2 \phi_i^l \\
+ \sum_{i \neq j \in [V]} \cos \phi_i^k \cos \phi_j^k \prod_{l = 1}^{k-1} \sin \phi_i^l \sin \phi_j^l \bigg ) \\
  - \bigg (\sum_{i=1}^V \prod_{l=1}^{n-1} \sin^2 \phi_i^l + \sum_{i \neq j \in [V]} \prod_{l = 1}^{n-1} \sin \phi_i^l \sin \phi_j^l \bigg ) 
\end{multlined}\\
&\!\begin{multlined}[t][12.2cm]
=
V^2 - \sum_{i=1}^V \Bigg ( \bigg ( \sum_{k=1}^{n-1} \cos^2 \phi_i^k \prod_{l=1}^{k-1} \sin^2 \phi_i^l \bigg ) + \bigg ( \prod_{l=1}^{n-1} \sin^2 \phi_i^l \bigg ) \Bigg )  \\
-\sum_{i \neq j \in [V]} \Bigg ( \bigg ( \sum_{k=1}^{n-1} \cos \phi_i^k \cos \phi_j^k \prod_{l = 1}^{k-1} \sin \phi_i^l \sin \phi_j^l \bigg ) \\
+ \bigg (\prod_{l = 1}^{n-1} \sin \phi_i^l \sin \phi_j^l \bigg ) \Bigg ).
\end{multlined}
\end{align*}
By the lemma, this equals
\[
V^2 - V - \sum_{i \neq j \in [V]} \Bigg ( \bigg ( \sum_{k=1}^{n-1} \cos \phi_i^k \cos \phi_j^k \prod_{l = 1}^{k-1} \sin \phi_i^l \sin \phi_j^l \bigg ) + \bigg ( \prod_{l=1}^{n-1} \sin \phi_i^l \sin \phi_j^l \bigg ) \Bigg ).
\]

For the left-hand side, we have:
\begin{align*}
\sum_{c \, \in \, \mathcal{C}}c^2&
\!\begin{multlined}[t][14.05cm]
= \frac{1}{2} \sum_{i \neq j \in [V]} \Big | \Big | \Big \langle \cos \phi_i^1, \cos \phi_i^2 \sin \phi_i^1, \ldots, \cos \phi_i^{n-1} \prod_{l=1}^{n-2} \sin \phi_i^l, \prod_{l=1}^{n-1} \sin \phi_i^l \Big \rangle \\ - \Big \langle \cos \phi_j^1, \cos \phi_j^2 \sin \phi_j^1, \ldots, \cos \phi_j^{n-1} \prod_{l=1}^{n-2} \sin \phi_j^l, \prod_{l=1}^{n-1} \sin \phi_j^l \Big \rangle \Big | \Big |^2 \\
\end{multlined}\\
&
\!\begin{multlined}[t][12.4cm]
= \frac{1}{2} \sum_{i \neq j \in [V]} \Bigg ( \sum_{k=1}^{n-1} \bigg ( \cos \phi_i^k \prod_{l = 1}^{k-1} \sin \phi_i^l - \cos \phi_j^k \prod_{l = 1}^{k-1} \sin \phi_j^l \bigg )^2 \\
+ \bigg ( \prod_{l=1}^{n-1} \sin \phi_i^l - \prod_{l=1}^{n-1} \sin \phi_j^l \bigg )^2 \Bigg )
\end{multlined}\\
&
\!\begin{multlined}[t][12.5cm]
=\frac{1}{2} \sum_{i \neq j \in [V]} \Bigg ( \sum_{k=1}^{n-1} \bigg ( \cos^2 \phi_i^k \prod_{l = 1}^{k-1} \sin^2 \phi_i^l \\
- 2 \cos \phi_i^k \cos \phi_j^k \prod_{l = 1}^{k-1} \sin \phi_i^l \sin \phi_j^l
+ \cos^2 \phi_j^k \prod_{l = 1}^{k-1} \sin^2 \phi_j^l \bigg ) \\
+ \bigg ( \prod_{l=1}^{n-1} \sin^2 \phi_i^l - 2 \prod_{l=1}^{n-1} \sin \phi_i^l \sin \phi_j^l + \prod_{l=1}^{n-1} \sin^2 \phi_j^l \bigg ) \Bigg ).
\end{multlined}
\end{align*}
Rearranging and applying the lemma twice, we have:
\begin{align*}
\MoveEqLeft \frac{1}{2} \sum_{i \neq j \in [V]} 1 + 1 - 2 \Bigg ( \bigg ( \sum_{k=1}^{n-1} \cos \phi_i^k \cos \phi_j^k \prod_{l = 1}^{k-1} \sin \phi_i^l \sin \phi_j^l \bigg ) + \bigg ( \prod_{l=1}^{n-1} \sin \phi_i^l \sin \phi_j^l \bigg ) \Bigg ) \\
&
\!\begin{multlined}[t][12.05cm]
= \sum_{i \neq j \in [V]} 1 - \Bigg ( \bigg ( \sum_{k=1}^{n-1} \cos \phi_i^k \cos \phi_j^k \prod_{l = 1}^{k-1} \sin \phi_i^l \sin \phi_j^l \bigg ) + \bigg ( \prod_{l=1}^{n-1} \sin \phi_i^l \sin \phi_j^l \bigg ) \Bigg ) \\
\end{multlined}\\
&
\!\begin{multlined}[t][13.05cm]
= V^2 - V - \sum_{i \neq j \in [V]} \Bigg ( \bigg ( \sum_{k=1}^{n-1} \cos \phi_i^k \cos \phi_j^k \prod_{l = 1}^{k-1} \sin \phi_i^l \sin \phi_j^l \bigg ) \\
 + \bigg ( \prod_{l=1}^{n-1} \sin \phi_i^l \sin \phi_j^l \bigg ) \Bigg ). \qedhere
\end{multlined}
\end{align*}
\end{proof}

In the second proof of Theorem \ref{Thm:sumoallchords}, we use the following definition and formulas involving the moment of inertia about a point (thanks to Petrov \cite{Petrov} for pointing out these formulas). The definition and both formulas can be found in \cite{BalkBoltyanskii}. For the remainder of this section, we identify points with the vectors pointing from the origin to them.

\begin{defn}
Let $\V = \{\mathbf{P_1}, \ldots, \mathbf{P_V}\}$ be a set of points.  Then the moment of inertia of $\V$ about a point $\mathbf{Q}$ is given by $J(\mathbf{Q})= \sum_{i=1}^V ||\mathbf{Q} - \mathbf{P_i}||^2$.
\end{defn}

\begin{lan}
Let $\V = \{\mathbf{P_1}, \ldots, \mathbf{P_V}\}$ be a set of points.  Then $J(\mathbf{Q}) = J(\mathbf{Z})+V||\mathbf{Q}-\mathbf{Z}||^2$ where $\mathbf{Z}$ is the centroid of $\V$.
\end{lan}

\begin{jac}
Let $\V = \{\mathbf{P_1}, \ldots, \mathbf{P_V} \}$ be a set of points.  Then $J(\mathbf{Z}) = V^{-1} \sum_{1 \leq i<j \leq V} ||\mathbf{P_i}-\mathbf{P_j}||^2$ where $\mathbf{Z}$ is the centroid of $\V$.
\end{jac} 

\begin{proof}[Second Proof of Theorem \ref{Thm:sumoallchords}]
We compute the moment of inertia about the center $\mathbf{O}$ of the $n$-sphere: $J(\mathbf{O})= \sum_{i=1}^V ||\mathbf{O} - \mathbf{P_i}||^2 =  \sum_{i=1}^V 1 = V$.      
Applying Langange's formula to $\mathbf{O}$, we have:  $V=J(\mathbf{O})=J(\mathbf{Z}) +V||\mathbf{Z}||^2$.  Substituting using Jacobi's formula, $V = 1/V \sum_{1 \leq i<j \leq V} || \mathbf{P_i} - \mathbf{P_j} ||^2 +V ||\mathbf{Z}||^2$.
This is the same as $V = 1/V \sum_{c \, \in \, \mathcal{C}}c^2 + Vd^2$.  Solving for the summation yields $\sum_{c \, \in \, \mathcal{C}}c^2 = V^2(1-d^2)$.
\end{proof}

Theorem \ref{Thm:sumoallchords} provides a new way to calculate a certain kind of potential energy function, namely, the \emph{frame potential}.  The frame potential is used by Benedetto and Fickus \cite{BenedettoFickus} to characterize a certain kind of desirable \emph{frame} known as a \emph{finite normalized tight frame}.  A \emph{frame for $\mathbb{R}^n$} is a spanning set of vectors $\{\mathbf{P_i} \}_{i=1}^V$ such that there exists constants $0 < A \leq B < \infty$ satisfying $A ||\mathbf{Q}||^2 \leq \sum_{i=1}^V |\langle \mathbf{Q}, \mathbf{P_i} \rangle |^2 \leq B ||\mathbf{Q}||^2$ for all $\mathbf{Q} \in \mathbb{R}^n$ \cite{Casazza}. 
Frames are used in certain applications (such as signal processing) where a basis could be a liability \cite{BenedettoFickus}.  (Thanks to Mixon \cite{Mixon} for alerting the author to a connection.)

\begin{defn}
The \emph{frame potential} of a sequence $\{\mathbf{P_i} \}_{i=1}^V$ in $\mathbb{R}^n$ is
\[ FP\big(\{\mathbf{P_i} \}_{i=1}^V \big) = \sum_{i=1}^V \sum_{j=1}^V |\langle \mathbf{P_i}, \mathbf{P_j} \rangle|^2 \]
\end{defn}

Benedetto and Fickus \cite{BenedettoFickus} focus on the case where the vectors $\mathbf{P_i}$ are unit vectors (since they study normalized frames).
  The following proposition uses Theorem $\ref{Thm:sumoallchords}$ to give
a formula for the frame potential in this important case.

\begin{prop}
\label{Prop:framepotential}
Let $\{\mathbf{P_i} \}_{i=1}^V$ be a sequence of points on the unit $n$-sphere centered at the origin in $\mathbb{R}^n$, and let $\mathbf{P_i}=(x_{1,i}, \ldots, x_{n, i})^T$.  Then
\[ FP\big(\{\mathbf{P_i} \}_{i=1}^V \big) = \bigg ( \sum_{j=1}^n \sum_{k=1}^V x_{j,k} \bigg )^2 \]
\end{prop}

Notice that, using the definition, it takes $2nV^2$ operations to compute the frame potential.  Using our Proposition \ref{Prop:framepotential} (which is based on Theorem \ref{Thm:sumoallchords}), it only takes $nV$ operations.

In our proof of Proposition \ref{Prop:framepotential}, we use the Frobenius norm on $n \times n$ matrices, i.e., the usual vector norm if one identifies the space of $n \times n$ matrices with $\mathbb{R}^{n^2}$ in the obvious way.  Hence, for an
 $n \times n$ matrix $A = [a_{ij}]$, we have:
  \[ || A ||_F = \sqrt{\sum_{i=1}^n \sum_{j=1}^n a_{ij}^2} \]
 
\begin{proof}
Notice that $ || \mathbf{P} \mathbf{P}^T  ||_F = 1 $ for any $ \mathbf{P}$ on the unit $n$-sphere:  Let $\mathbf{P} = (x_1, \ldots , x_n)^T$ be such a point, so $ \sqrt{x_1^2 + \cdots + x_n^2 } = 1$. 
Then $ \mathbf{P} \mathbf{P}^T $ has $(i, j)$-entry $x_i x_j$ and 
$ || \mathbf{P} \mathbf{P}^T  ||_F $ = $ \sqrt{ \sum_{i=1}^n \sum_{j=1}^n (x_i x_j)^2 } $
= $ \sqrt{ \Big ( \sum_{i=1}^n x_i^2 \Big )^2 } = 1$.

We follow Mixon \cite{Mixon} in relating the sum of squared distances
between the points $\mathbf{P_iP_i}^T$ to the frame potential. 
 First, we show $ || \mathbf{P} \mathbf{P}^T - \mathbf{Q} \mathbf{Q}^T ||_F^2 = 2 - 2(\mathbf{P}^T \mathbf{Q})^2 $ 
for all $ \mathbf{P}, \mathbf{Q} $ on the unit $n$-sphere.
Let $ \mathbf{P} = (x_1, \ldots , x_n)^T$ and $ \mathbf{Q} = (y_1, \ldots , y_n)^T$ be such points.
Then
\begin{align*}
\!\begin{multlined}[t][13.75cm]
 || \mathbf{P} \mathbf{P}^T - \mathbf{Q} \mathbf{Q}^T ||_F^2 
 = \sum_{i=1}^n \sum_{j=1}^n (x_i x_j - y_i y_j)^2 = 
 \sum_{i=1}^n \sum_{j=1}^n (x_i x_j)^2 \\ - 
 \sum_{i=1}^n \sum_{j=1}^n 2 x_i x_j y_i y_j 
 + \sum_{i=1}^n \sum_{j=1}^n (y_i y_j)^2  = 
 || \mathbf{P} \mathbf{P}^T  ||_F^2  \\ - 
 2 \Big ( \sum_{i=1}^n x_i y_i \Big )^2  + || \mathbf{Q} \mathbf{Q}^T  ||_F^2  = 2 - 2(\mathbf{P}^T \mathbf{Q})^2.
 \end{multlined}
\end{align*} 
Thus:
\begin{align*}
\!\begin{multlined}[t][13.75cm]
\sum_{i=1}^V \sum_{j=1}^V || \mathbf{P_i} \mathbf{P_i}^T - \mathbf{P_j} \mathbf{P_j}^T ||_F^2 \\ = 2 V^2 - 2 \sum_{i=1}^V \sum_{j=1}^V (\mathbf{P_i}^T \mathbf{P_j} )^2 
= 2V^2 - 2 FP\big(\{\mathbf{P_i} \}_{i=1}^V \big).
 \end{multlined}
\end{align*} 

Now we
can apply Theorem 2.1.  If we view the matrices $\{ \mathbf{P_i} \mathbf{P_i}^T \}_{i=1}^V $ as points in $\mathbb{R}^{n^2}$, then 
Theorem \ref{Thm:sumoallchords} shows that
\begin{align*}
\!\begin{multlined}[t][13.75cm]
 2 V^2(1-d^2) = \sum_{j=1}^V \sum_{i=1}^V || \mathbf{P_i} \mathbf{P_i}^T - \mathbf{P_j} \mathbf{P_j}^T ||_F^2 = 2V^2 - 2 FP\big(\{\mathbf{P_i} \}_{i=1}^V \big)
 \end{multlined}
\end{align*}
  where $d$ is the Frobenius norm of the centroid of the points $ \{ \mathbf{P_i} \mathbf{P_i}^T \}_{i=1}^V.$
 Hence, $FP\big(\{\mathbf{P_i} \}_{i=1}^V \big) = V^2 d^2 = V^2 || \frac{1}{V} \sum_{i=1}^V \mathbf{P_i} \mathbf{P_i}^T ||_F^2$.  
 If we let $ \mathbf{P_i} = \\ (x_{1,i}, \ldots, x_{n, i})^T$, then the $(j, k)$-entry of $\mathbf{P_i} \mathbf{P_i}^T$ is $x_{j, i}x_{k, i}$ and 
 $\\ V^2 || \frac{1}{V} \sum_{i=1}^V \mathbf{P_i} \mathbf{P_i}^T ||_F^2 
 = \sum_{k=1}^n \sum_{j=1}^n \big ( \sum_{i=1}^V x_{j, i} x_{k, i} \big )^2 =
 \big ( \sum_{j=1}^n \sum_{k=1}^V x_{j,k} \big )^2$. \qedhere
\end{proof}

\section{Sum of Squared Distinct Distances}

Before we generalize Fact 2, we introduce a little terminology to describe the symmetry of a set of points.  We say that the symmetry group $G$ of $\V$ acts \emph{transitively} on $\V$ if, for every $\mathbf{P}, \mathbf{Q} \in \V$, there exists $g \in \V$ such that $g \mathbf{P} = \mathbf{Q}$.

\begin{defn}
A set of points $\mathcal{V}$ is \emph{symmetric about the origin} if it is closed under the antipodal map, i.e., for every point $\mathbf{P}$ in $\mathcal{V}$, the point $-\mathbf{P}$ is in $\mathcal{V}$ (cf. \cite{Coxeter1973,Grunbaum}).
\end{defn}

Now we generalize Fact 2.

\begin{thm}
\label{Thm:distinctsymm}
Let $\V$ be a finite set of points on a unit $n$-sphere centered at the origin in $\mathbb{R}^n$, and let $\mathcal{D}$ be the set of distinct lengths of the chords between them.  Suppose further that $\V$ is symmetric about the origin and the symmetry group of $\V$ acts transitively on $\V$.
Then: \[ \sum_{d \, \in \, \mathcal{D}}d^2= 2k+2 \]
where $k$ is the cardinality of $\mathcal{D}$.
\end{thm}

\begin{proof}[Proof of Theorem \ref{Thm:distinctsymm}]
Choose a point $\mathbf{P}_{\mathbf{0}} \in \mathcal{V}$.
By transitivity, there exist points $\mathbf{P_{i}} \in \V$ for $i=1,2,\ldots,k$ such that the distance from $\mathbf{P}_{0}$ to $\mathbf{P_{i}}$
is $d_i$ and $d_{1}<d_{2}<\cdots<d_{k}$.
Let $e_i$ be the distance from $\mathbf{P_0}$ to $-\mathbf{P_i}$ and denote $-\mathbf{P_i}$ by $\mathbf{Q_i}$ for $i=1, 2, \ldots, k$.
Since $\V$ is symmetric about the origin, $\mathbf{Q_i} \in \V$ for $i=1, 2, \ldots, k$.
Note that $\mathbf{P_0P_k}$ must be a diameter of the $n$-sphere, so $\mathbf{Q_k = P_0}$.
  Since $\mathbf{P_iQ_i}$ is also a diameter of the $n$-sphere, the triangle $\mathbf{P_iP_0Q_i}$  has a right angle at $\mathbf{P_0}$ for $i=1,2 \ldots, k-1$.
Moreover, since all the triangles have the same hypotenuse length of 2 and $d_1 < \cdots < d_{k-1}$, we have $d_i^2 +e_i^2 = 2^2$ for $i=1,2 \ldots, k-1$ by the Pythagorean Theorem and so $e_1 > \cdots > e_{k-1}$.  Note that we must have $e_1 < 2$ because otherwise $\mathbf{P_0Q_1}$ would be a diameter and we would have $\mathbf{P_0} = \mathbf{P_1}$.
Thus, $\{ d_1, \ldots, d_{k-1} \} = \{e_1, \ldots, e_{k-1} \}$.  Now:

\[ \sum_{i=1}^{k-1} 2d_i^2 = \sum_{i=1}^{k-1} (d_i^2 + e_i^2) = \sum_{i=1}^{k-1} 2^2 = 4k - 4. \]
Hence:
\[ 2 \sum_{d \in \mathcal{D}} d^2 = \sum_{i=1}^k 2d_i^2 = \sum_{i=1}^{k-1} 2d_i^2 + 2(2^2) = 4k + 4, \]
which implies
\[ \sum_{d \in \mathcal{D}} d^2 = 2k + 2. \qedhere \] 
\end{proof}

Since Theorem \ref{Thm:distinctsymm} only holds for point sets that are symmetric about the origin, we also establish bounds on the summation in the case where the point set $\V$ is not symmetric about the origin.  We use antipodal symmetrization to achieve this.

\begin{defn}
Given a set $\V$ of points, define its \emph{antipodal symmetrization} to be the set $\V \cup -\V$ where $-\V$ denotes the set $ \{ -\mathbf{P} : \mathbf{P} \in \V \}$.  Denote this antipodal symmetrization by $\text{AntiSym}(\V)$.
\end{defn}

\begin{prop}
\label{Prop:nonsymmbounds}
 Let $\V$ be a finite set of points on a unit $n$-sphere centered at the origin in $\mathbb{R}^n$, and let $\mathcal{D}$ be the set of the distinct lengths of the chords between them.
 Suppose further that $\V$ is not symmetric about the origin, its symmetry group acts transitively on $\V$, and the only point in $\mathbb{R}^n$ fixed by the symmetry group is the center of the $n$-sphere.
Then \[ 2k-2r \leq \sum_{d \, \in \, \mathcal{D}}d^2 \leq 2k+2r \]
where $k$ is the cardinality of $\mathcal{D}$ and $r$ is the number of non-diameter chord lengths of $\text{AntiSym}(\V)$ that are \emph{not} chord lengths of $\V$.
Furthermore, $r \leq k$.
\end{prop}  

\begin{proof}[Proof of Proposition \ref{Prop:nonsymmbounds}]
For the upper bound:  First, note that $\text{AntiSym}(\V)$ has a chord that is a diameter of the $n$-sphere since $\text{AntiSym}(\V)$ is symmetric about the origin.  This chord has length 2.  Note also that $\V$ does not have a chord that is a diameter (else, by transitivity, $\V$ would be symmetric about the origin).

Let $\mathcal{R}$ be the set of non-diameter chord lengths of $\text{AntiSym}(\V)$ that are not in $\mathcal{D}$, so $| \mathcal{R} | = r$. 
 Since the distinct chord lengths of $\text{AntiSym}(\V)$ are those in $\mathcal{D} \cup \mathcal{R}$ along with a diameter, the total number of distinct chords of $\text{AntiSym}(\V)$ is $k+r+1$.
Moreover, since $\text{AntiSym}(\V)$ is symmetric about the origin, Theorem \ref{Thm:distinctsymm} shows that
 \[ \sum_{d \, \in \, \mathcal{D} \cup \mathcal{R}}d^2 + 2^2 = 2(k+r+1)+2 \]
Subtracting 4 from both sides and noting that $\sum_{d \, \in \, \mathcal{R}} d^2 \geq 0$ shows us that 
\[ \sum_{d \, \in \, \mathcal{D}} d^2 \leq  \sum_{d \, \in \, \mathcal{D}} d^2 + \sum_{d \, \in \, \mathcal{R}} d^2 = 2k+2r \]

For the lower bound:  Choose a point $\mathbf{P_0} \in \V$. By transitivity, there exist points $\mathbf{P_{i}} \in \V$ for $i=1,2,\ldots,k$ such that $\mathbf{P_0 P_1}, \ldots, \mathbf{P_0 P_k}$ represent all the distinct chord lengths of $\V$.
Let $\mathbf{Q_i}$ denote $-\mathbf{P_i}$ for $i=0,1, \ldots, k$.  Note that $\mathbf{P_0Q_0}$ is a diameter.

We claim that each chord length in $\mathcal{R}$ is represented by $\mathbf{P_0Q_i}$ for some $1 \leq i \leq k$ (which implies $r \leq k$).  Suppose $\mathbf{P_0Q}$ where $\mathbf{Q} \not \in \{ \mathbf{Q_1}, \ldots, \mathbf{Q_k} \}$ is a non-diameter chord of $\text{AntiSym}(\V)$ whose length is not in $\mathcal{D}$. 
Then $\mathbf{Q}=-\mathbf{P}$ for some $\mathbf{P} \in \V$ and $\mathbf{P} \mathbf{P_0} \mathbf{Q}$ and $\mathbf{P_i} \mathbf{P_0} \mathbf{Q_i}$  are right triangles with
diameter-long hypotenuses $\mathbf{PQ}$ and $\mathbf{P_iQ_i}$, respectively.  By the Pythagorean theorem, the lengths of sides $\mathbf{P_0Q}$ and $\mathbf{P_0Q_i}$ are equal if and only if the lengths of sides $\mathbf{P_0P}$ and $\mathbf{P_0P_i}$ are equal.  Hence, there must be an $i$ such that $\mathbf{P_0P_i}$ is the same length as $\mathbf{P_0P}$.

Renumber the $\mathbf{Q_i}$ so that $\mathbf{P_0Q_1}, \ldots, \mathbf{P_0Q_r}$ are the chords whose lengths are in $\mathcal{R}$ and renumber the $\mathbf{P_i}$ so that $\mathbf{P_i} = - \mathbf{Q_i}$ for $1 \leq i \leq r$.  Denote the length of $\mathbf{P_0 Q_i}$ by $e_i$ for $0\leq i \leq r$ and the length of $\mathbf{P_0 P_i}$ by $d_i$ for $1 \leq i \leq r$.
Applying the Pythagorean Theorem to the triangles $\mathbf{P_1} \mathbf{P_0} \mathbf{Q_1}, \ldots,\mathbf{P_r} \mathbf{P_0} \mathbf{Q_r}$ shows that $\sum_{i=1}^rd_i^2 + \sum_{i=1}^r e_i^2 = 4r$.
Applying Theorem \ref{Thm:distinctsymm} to $\text{AntiSym}(\V)$ shows that $\sum_{i=1}^k d_i^2 + e_0^2 + \sum_{i=1}^r e_i^2 = 2(k+r+1)+2$.  Subtracting the previous equation from this yields  $\sum_{i=r+1}^k d_i^2 + e_0^2 = 2k-2r+4$.  Hence, $\sum_{i=r+1}^k d_i^2 = 2k-2r$.  Thus, $2k-2r \leq \sum_{i=1}^k d_i^2 = \sum_{d \, \in \, \mathcal{D}} d^2$.
\end{proof}

\begin{rem}
Maple was used to compute $\sum_{d \in \mathcal{D}} d^2$, $r$, and $k$ for several point sets satisfying the hypotheses of Proposition \ref{Prop:nonsymmbounds}: the vertices of regular simplices; prisms and antiprisms with odd-edged and even-edged regular polygon bases, respectively; the convex hulls of $W \mathbf{P}$ and $R \mathbf{P}$ where $W$ is a Coxeter group, $R$ its subgroup of rotations, and $\mathbf{P}$ a random unit vector; and permutahedra.  In most cases, the value of $r$ was actually $k$ and so the lower and upper bounds given in Theorem \ref{Prop:nonsymmbounds} amounted to a trivial 0 and $4k$, respectively.

However, certain prisms and antiprisms did have $r < k$ and so the bounds given in Theorem \ref{Prop:nonsymmbounds} are not inherently trivial.  For instance, consider the vertices $\V$ of a prism inscribed inside a unit 3-sphere centered at the origin $\mathbf{O}$ whose bases are equilateral triangles such that the vertices in the bases have $z$-coordinates of $\pm 1/ \sqrt{2}$ (using the standard Cartesian coordinate system), i.e., they are located at an angle of $\pi/4$ from the positive and negative $z$-axes.  Observe that $\text{AntiSym}(\V)$ is a prism just like $\V$ except with regular hexagon bases.

Fix a point $\mathbf{P_0}$ of $\V$ in the ``top base.''  Let $\mathbf{P_1}$ be a different point in the ``top base,'' let $\mathbf{P_2}$ be the point below
 $\mathbf{P_1}$, and let $\mathbf{P_3}$ be the point below $\mathbf{P_0}$.
From the mirror symmetry of $\V$, one can see that $\V$ has at most three
 distinct chord lengths, namely, those represented by the chords $\mathbf{P_0P_1}$, 
 $\mathbf{P_0P_2}$, and $\mathbf{P_0P_3}$.  Applying the Law of Cosines to the triangles $\mathbf{P_0OP_3}$, $\mathbf{P_0P_3P_2}$, and $\mathbf{P_0P_1P_2}$ (in that order), one can show that these chords have the distinct lengths $\sqrt{3/2}$,
  $\sqrt{7/2}$, and $\sqrt{2}$, respectively.
 
 By the proof of Theorem \ref{Thm:distinctsymm}, the nondiameter distinct chord lengths of $\text{AntiSym}(\V)$ that are not chord lengths of $\V$ can be represented by some subset of $\mathbf{P_0Q_1}$, $\mathbf{P_0Q_2}$, $\mathbf{P_0Q_3}$ where $\mathbf{Q_i = - P_i}$  for each $i$.  To show that $r < k=3$, we show that one of these, namely, $\mathbf{P_0P_3}$, is the same length as one of the chords of $\V$, namely, $\mathbf{P_0Q_3}$.
  
The vectors pointing from the origin to the points $\mathbf{P_0}$ and $\mathbf{Q_3}$ are both located at $\pi/4$ from the positive $z$-axis, so the angle between them is $\pi/ 2$.  By the Law of Cosines applied to triangle $\mathbf{P_0OQ_3}$, the chord $\mathbf{P_0Q_3}$  has length $\sqrt{1+1-2\cos(\pi/2)} = \sqrt{2}$.
Since the vector pointing from the origin to $\mathbf{P_3}$ is $3 \pi /4$ from the positive $z$-axis, the angle between this vector and the one pointing to $\mathbf{P_0}$ is $3 \pi/4 -\pi/4 = \pi/2$.  Thus, by the Law of Cosines, this chord also has length $\sqrt{2}$.
\end{rem}

\begin{rem} Perhaps a more realistic idea as to what the exact upper bound on $\sum_{d \in \mathcal{D}} d^2$ might be can be inferred from the fact that the highest calculated value (in terms of $k$) was $2k+2 \cdot 2$ (where $r$ was not 2 but rather $16$).
\end{rem}
 
\section*{Acknowledgements}
Many thanks to Nathan Reading for his guidance, ideas, and encouragement concerning this research.

\end{document}